\documentclass[12pt]{article}
\usepackage{amssymb}
\usepackage{amsmath}
\usepackage{hyperref}
\usepackage{doublespace}
\usepackage{graphicx}

\newtheorem{theorem}{Theorem}

\newtheorem{definition}{Definition}

\newtheorem{proposition}{Proposition}
\newtheorem{remark}{Remark}

\newenvironment{proof}[1][Proof]{\noindent\textbf{#1.} }{\ \rule{0.5em}{0.5em}}
\input{tcilatex}

\begin{document}

\title{ON DEGENERATE PLANAR HOPF BIFURCATIONS}
\author{M. R. Ricard \\
Faculty of Mathematics and Computer Science, \\
University of Havana, \\
CP.10400, Cuba\\
rricard@matcom.uh.cu}
\date{}
\maketitle

\begin{abstract}
Our concern is the study of degenerate Hopf bifurcation of smooth planar
dynamical systems near isolated singular points. To do so, we propose to
split up the definition of degeneracy into two types. Degeneracy of first
kind shall means that no limit cycle surrounding the steady state can emerge
after or before the critical point, with the possible emergence of limit
cycles surrounding the point at infinity. Degeneracy of second kind shall
means that either several limit cycles or semistable cycles as a limiting
case, emerge surrounding the steady state super or subcritically. In
degenerate bifurcation of second kind we also show that the radius of the
emerging cycle tends to zero with an "anomalous" order as the bifurcation
parameter tends to the critical value. Finally, we give a sufficient
condition for degenerate bifurcations of second kind up to $6$%
-jet-equivalence, and show some \textquotedblleft typical\textquotedblright\
forms for degenerate bifurcations.


\end{abstract}

\section{Introduction}

The goal of this paper is the study of degenerate Hopf bifurcation (\textbf{%
HB}) near the critical value of a one-parameter family of smooth planar
vector fields in a neighborhood of an isolated singular point. The HB is
usually associated to the emergence of limit cycles after or before the
critical value of the bifurcation parameter. An isolated periodic orbit of a
vector field on the plane, is a limit cycle. The HB phenomenon is studied
recurrently in literature, and its main details and implications for
dynamical systems of higher dimensions via the Center Manifold Theorem are
gathered in the monograph \cite{marsden}. Several efforts to compute
generalized HB in finite dimensional systems can be found in literature (see 
\cite{gross} and references therein). Applications of the HB in finite
dimensional dynamical systems often arise in the study of population-based
models, for instance in predator-prey systems featured by the Allee effect 
\cite{aguirre}; in chemical systems, as in the Schnakenberg's \cite%
{edelstein}; in the study of coupled systems near a supercritical HB \cite%
{drover}; in the theory of electronic circuits \cite{chen}; in mathematical
economics \cite{gandolfo}; in the modeling of mechanical systems \cite%
{sotomayorNA}; even in the study of travelling waves phenomena \cite%
{murrayII} and, generally speaking, almost whenever the considered model
shows non-linear oscillations. Remarkable applications of HB in fluid
dynamics can be found in \cite{marsden} and references therein. In addition,
the topic is extensively treated in the theory of bifurcations with higher
codimension, as in the Turing-Hopf instabilities \cite{mrrsm}. Nevertheless,
in order to focalize the scenario, we suggest in this paper to split the
classification of degeneracy into two different kinds. Degeneracy of first
kind shall means that$\ $no limit cycle surrounding the steady state emerges
neither after nor before the critical value of the bifurcation parameter,
with the possible emergence of limit cycles surrounding the point at
infinity \cite{gasull}. The second kind of degeneracy shall means that
multiple limit cycles or semistable cycles as a limiting case, emerge either
at super- or at subcritical bifurcation surrounding the steady state. We
focus our attention in such a system that, in a neighborhood of the origin,
can be written in the form 
\begin{equation}
\overset{\cdot }{X}=F\left( X,a\right)  \label{variations}
\end{equation}%
where 
\begin{equation}
F=\dsum\limits_{k+l=1}^{R}\left( 
\begin{array}{c}
\sigma _{kl}^{1}\  \\ 
\sigma _{kl}^{2}%
\end{array}%
\right) x^{k}\ y^{l}+O\left( \left\Vert X\right\Vert ^{R+1}\right)  \label{F}
\end{equation}%
is the Taylor expansion of $F$ in a neighborhood of the origin, the
coefficients $\sigma _{kl}^{m}=\sigma _{kl}^{m}\left( a\right) $ are real,
and $X=\left( x,y\right) ^{T}$. Here $a$ is a real parameter, and $F$ is at
least a $C^{0}$-function of $a$. Any smooth dynamical system in a vicinity
of an isolated singular point $P_{a}$ on the plane, can be rewritten in the
form Eq.\ref{variations} by a parameter-dependent shift of coordinates which
relocates the singular point $P_{a}$ of the system into the origin of
coordinates. Eq.\ref{variations} is called the system \emph{in variations}
corresponding to a given former system near the singular point $P_{a}$. Any
reference to the steady state $P_{a}$ of the former system will be
associated to the punctual orbit at the origin of the Eq.\ref{variations}.
The vector field with polynomial components of degree $R$ represented in the
main part in Eq.\ref{F}, is called the $R$\emph{-jet }of Eq.\emph{\ref%
{variations} }around the origin, which is usually denoted $j_{R}\left(
F\right) \left( 0\right) $. Two$\ $smooth vector fields defined in a
neighborhood of the origin are called $R$\emph{-jet equivalent} if their $R$%
-jet coincide.

Dynamical systems featured by an HB often arise in mathematical models for
biological or chemical reaction systems \cite{edelstein}. In general, these
models lead to polynomial vector fields having fifth degree at most.
Moreover, some complex interactions between the different components in the
system may lead to a representation in which a rational fraction with linear
or quadratic denominator appear in the former system. For instance, this is
the case in the Michaelis-Menten kinetics or in the prey-predator Holling
systems in presence of Allee effect \cite{aguirre}. In such situations,
after rescaling the temporal variable, the system can be transformed into an
orbitally equivalent \cite{kuznetsov} polynomial system of degree at most
six. {}

We recall that only few combinations of coefficients corresponding to odd
terms in the reaction system can contribute to the emergence of limit cycles
at bifurcation. We have called these the \emph{Hopf coefficients}. These
Hopf coefficients can be easily calculated and will play a similar role than
such the Lyapunov coefficients do \cite{sotomayor}. We suggest a new
classification of degenerate bifurcations in terms of the Hopf coefficients,
through the \textquotedblleft discriminant\textquotedblright\ introduced in 
\cite{mrrsm}. In this direction, we introduce the concepts of degenerate HB
of first and second kind.

In accordance with the proposed classification we subdivide the HB Theorem,
in order to emphasize differences in the resulting behavior at bifurcation.
We prove that, at degenerate bifurcation of first kind no limit cycle
surrounding the steady state emerges, being allowed the emergence of limit
cycles surrounding the point at infinity. At supercritical (subcritical)
degenerate HB of second kind multiple or semistable limit cycles emerge
surrounding the steady state. In correspondence to the procedure in \cite%
{mrrsm}, we shall see that, at a non-degenerate HB, a single limit cycle
emerges. Each kind of degenerate HB\ is complemented with appropriate
examples.

The period of the emerging limit cycles in one-parameter bifurcations
attract the attention of researchers (see \cite{gasull}) showing that, the
main term in the asymptotic of the period of the emerging periodic solution,
characterizes the bifurcation. Here we focus our attention in a procedure
that gives us simultaneously an asymptotic estimate of the radius of the
cycle and the frequency of the corresponding periodic orbit. As we shall
see, the\ radius of the limit cycles and the period of the periodic
solutions emerging at degenerate HB of second kind, tends to zero with an
\textquotedblleft anomalous\textquotedblright\ order as the bifurcation
parameter tends to the critical value.

The plan of the paper is as follows. In Section \ref{Section Preliminaries}
we summarize previous results in the treatment of an HB taken basically from 
\cite{mrrsm}. In Section \ref{Section classification}, the notion of
discriminant is quoted, which is used in Definition \ref{definition
bifurcations} to classify HB taking into account its asymptotic behavior as
the trace of the Jacobian tends to zero. The splitted version of the HB
Theorem in \ref{hopf th}, \ref{hopf degenerate 1st} and \ref{hopf degenerate
th} in this Section describe the emergence of limit cycles in
non-degenerate, degenerate of first and of second kind respectively.
Finally, in Section \ref{SectionNormalForms} are gathered "typical" forms
corresponding to each type of degenerate bifurcation. We include a
sufficient condition for degenerate HB of second kind of Eq.\ref{variations}
valid up to $6$-jet-equivalence .

\section{Preliminaries\label{Section Preliminaries}}

Let $\delta _{a}$ and $\tau _{a}$ be the determinant and the trace 
\begin{eqnarray}
\delta _{a} &=&\det \left( J_{a}\right)  \label{jacobian} \\
\tau _{a} &=&\func{trace}\left( J_{a}\right)  \label{trace jacobian}
\end{eqnarray}%
of the Jacobian matrix 
\begin{equation}
J_{a}=\left( 
\begin{array}{cc}
\sigma _{10}^{1} & \sigma _{01}^{1} \\ 
\sigma _{10}^{2} & \sigma _{01}^{2}%
\end{array}%
\right)  \label{Ja}
\end{equation}%
at the origin $O$ of the function $F$ in Eq.\ref{variations}. The subindex
in Eqs.\ref{jacobian} and \ref{trace jacobian} indicates a functional
dependence respect to $a$, which varies in an open small neighborhood $U$ of
the point $a_{\ast }$ at which the trace vanish to change its sign. We shall
assume that, the function $a\leadsto \tau _{a}$ is an homeomorphism between $%
U$ and a neighborhood $V$ of $\tau _{a}=0$, so the transversality condition $%
\tau _{a}^{\prime }\left( a_{\ast }\right) \neq 0$ is not required.
Consequently, the parameter $\tau _{a}$, $\left\vert \tau _{a}\right\vert
\ll 1$, can be considered as the \emph{intrinsic} bifurcation parameter (see 
\cite{mrrsm}) and, $\tau _{a}=0$ is the critical value. The HB appears
provided the inequality 
\begin{equation}
\tau _{a}^{2}-4\delta _{a}<0  \label{complex eigenvalues cond}
\end{equation}%
holds for any value $\tau _{a}$ in some neighborhood of $\tau _{a}=0$. Then,
for every value $a\in U$, or equivalently, for every value $\tau _{a}\in V$,
we have 
\begin{equation}
\delta _{a}\ >0\text{ .}  \label{positive jacobian}
\end{equation}%
Further, if $\tau _{a}<0$ (respect. $\tau _{a}>0$) the origin is a stable
(respect. unstable) focus. The bifurcation is \emph{subcritical} if there is
a limit cycle emerging for negative values of $\tau _{a}$ close enough to
zero. The bifurcation is \emph{supercritical} if there is a limit cycle
emerging for positive small values of $\tau _{a}$. From Eq.\ref{complex
eigenvalues cond} follows $\sigma _{01}^{1}\cdot \sigma _{10}^{2}<0$ and we
do not loose generality assuming $\sigma _{01}^{1}<0$. Besides, there is no
added restriction if we consider the Jacobian matrix in the system Eq.\ref%
{variations} to have the simplest form:%
\begin{equation}
J_{a}=\frac{1}{2}\left( 
\begin{array}{cc}
\tau _{a} & -\Lambda _{a} \\ 
\Lambda _{a} & \tau _{a}%
\end{array}%
\right)  \label{J}
\end{equation}%
where 
\begin{equation*}
\Lambda _{a}=+\sqrt{4\delta _{a}-\tau _{a}^{2}}>0
\end{equation*}%
in a neighborhood of $\tau _{a}=0$. So, we shall assume in the following
that the Jacobian in Eq.\ref{Ja} has the form Eq.\ref{J}. If necessary, the
system Eq.\ref{variations} can be rewritten through a linear transformation
of variables in order that the Jacobian has the required form Eq.\ref{J},
whenever the condition Eq.\ref{complex eigenvalues cond} holds.

\subsection{Averaging Hopf periodic solutions\label{subsection hopf}}

Let us quote in this Subsection some results about the procedure in the
study of HB proposed in \cite{mrrsm}. Following that paper, we rewrite the
system Eq.\ref{variations}, in the form%
\begin{equation}
\overset{\cdot }{X}=J_{a}X+\Psi \left( X\right)  \label{linearization}
\end{equation}%
where $X=\left( x\left( t\right) ,y\left( t\right) \right) ^{T}$, and the
vector function $\Psi $ is given by the Taylor expansion of the difference 
\begin{equation*}
\Psi \left( X\right) =F\left( X,a\right) -J_{a}X
\end{equation*}%
so, $\Psi \left( X\right) $ contains all nonlinearities. In the function $%
\Psi \left( X\right) $ we include the Taylor terms until a required
precision, say up to $R$-jet-equivalence, together with the corresponding
remainder. First we assume that $\Psi \left( X\right) $ is analytical and a
bit later, in Remark \ref{smooth}, we turn back into the smooth case. In 
\cite{mrrsm}, the authors proposed an algorithm allowing the reduction of an
analytical reaction system showing an HB into a second order differential
equation representing a weakly nonlinear oscillator in normal form. This
transform of variables is analytical and nonlinear in general, but it was
proved there that it is enough to consider the linear part of this transform
to obtain the equation of the oscillator preserving the required accuracy.
The main idea of this procedure is that the transform of variables can be
taken \textquotedblleft close\textquotedblright\ to the appropriate linear
transform in a neighborhood of the origin. So we quote the procedure in \cite%
{mrrsm}, and consider an invertible analytical transform of variables
between neighborhoods of the origin 
\begin{equation}
Y=\mathcal{H}\left( X\right) =\Gamma X+\mathcal{G}\left( X\right)
\label{Y=Gamma X}
\end{equation}%
where $\Gamma =\left( \gamma _{ij}\right) $ is a non-singular matrix, and $%
\mathcal{G}\left( X\right) $ be analytical without linear terms. By the
Inverse Function Theorem, the existence of the inverse is guaranteed because 
$\Gamma $ is non-singular and $\mathcal{H}$ has smooth continuous
derivatives. The inverse to Eq.\ref{Y=Gamma X} has the form 
\begin{equation}
X=\mathcal{H}^{-1}\left( Y\right) =\Gamma ^{-1}Y+\mathcal{K}\left( Y\right) 
\text{ .}  \label{inverse change}
\end{equation}

\begin{definition}
We say that the diffeomorphisms $\mathcal{H}$ \emph{and }\textbf{$\Gamma $} 
\emph{have a contact at the origin of order} $S\in 
\mathbb{N}
$, $S\geq 1,$ if 
\begin{equation}
\mathcal{H}\left( X\right) -\Gamma X=O\left( \left\Vert X\right\Vert
^{S+1}\right)  \label{N contact}
\end{equation}%
as $\left\Vert X\right\Vert \rightarrow 0$.
\end{definition}

We would like to substitute the function $\mathcal{H}$ in Eq.\ref{Y=Gamma X}
by an equivalent simpler one, say $\Gamma $, in the procedure of rewriting
the system Eq.\ref{variations} in new easy-handled variables to be used in
Proposition \ref{Prop gamma generators}. In this endeavor we get more
precision as the order of contact $S$ between $\mathcal{H}$ and $\Gamma $ be
greater. Further, from Eq.\ref{N contact} follows that $\mathcal{H}^{-1}$
exists whenever $\Gamma $ is invertible, and it can be also concluded that $%
\mathcal{H}^{-1}$ and $\Gamma ^{-1}$have a contact of order $S$. Following
the ideas in \cite{mrrsm} it can be proved that, if $\Gamma $ satisfies
certain \textquotedblleft concordance\textquotedblright\ condition and $%
\mathcal{H}$ has a contact with $\Gamma $ at the origin of order $R$, then $%
Y=\mathcal{H}\left( X\right) $ represents an analytical transform of
coordinates such that every solution $\left( x\left( t\right) ,y\left(
t\right) \right) ^{T}$ to the analytical system Eq.\ref{variations} is
transformed into the form 
\begin{equation}
Y=\left( z\left( t\right) ,\overset{\cdot }{z}\left( t\right) \right) ^{T}
\label{Y}
\end{equation}%
being $z\left( t\right) $ unknown. Hence, the integration of the system Eq.%
\ref{linearization} can be reduced to the integration of a second order
differential equation in the variable $z$. Consequently, rather than the
exact expressions of the functions $\mathcal{G}\left( X\right) $ and $%
\mathcal{K}\left( Y\right) $, it is sufficient to take the linear transforms 
$\Gamma $ and $\Gamma ^{-1}$ instead of $\mathcal{H}$ and $\mathcal{H}^{-1}$
in the derivation of the equation in $z$, considering that this substitution
preserves the required accuracy. More precisely, for a given $R$ in Eq.\ref%
{F} we may take a transform of coordinates $\mathcal{H}$ which has a contact
of order $R$ with $\Gamma $ at the origin. In the following statement we
quote a result in \cite[Prop.1]{mrrsm} in which are gathered all the above
ideas:

\begin{proposition}
\label{Prop gamma generators}Let us assume that Eq.\ref{complex eigenvalues
cond} holds. Then, there exists an invertible analytical transform of
variables Eq.\ref{Y=Gamma X} in the system Eq.\ref{linearization} such that
Eq.\ref{Y} holds. The matrix $\Gamma $ is any non-trivial linear combination
of the pair 
\begin{equation}
\Gamma _{1}=\left( 
\begin{array}{cc}
1 & 0 \\ 
\sigma _{10}^{1} & \sigma _{01}^{1}%
\end{array}%
\right) \text{ ; }\Gamma _{2}=\left( 
\begin{array}{cc}
0 & 1 \\ 
\sigma _{10}^{2} & \sigma _{01}^{2}%
\end{array}%
\right) \text{ .}  \label{gamma generators}
\end{equation}%
The function $z$ in Eq.\ref{Y} satisfies the following second order
equation: 
\begin{equation}
\overset{\cdot \cdot }{z}-\ \tau _{a}\overset{\cdot }{z}+\ \delta _{a}\
z=G\left( z,\overset{\cdot }{z}\right)  \label{second order general}
\end{equation}%
where the right hand side in Eq.\ref{second order general} does not involve
linear terms in $z$, $\overset{\cdot }{z}$. More precisely, 
\begin{equation}
G\left( z,\overset{\cdot }{z}\right) =\Pi _{2}\left\{ \Gamma \left( J_{a}%
\mathcal{K}\left( Y\right) +\Psi \left( \mathcal{H}^{-1}Y\right) \right)
\right.  \label{G}
\end{equation}%
\begin{equation*}
\left. +\left\langle \mathtt{grad}_{X}\ \mathcal{G\ },\mathcal{F}%
\right\rangle \left( \mathcal{H}^{-1}Y\right) \right\}
\end{equation*}%
being $\Pi _{2}$ the standard projector over the second component.
\end{proposition}

So, the function in Eq.\ref{G} can be expanded asymptotically by%
\begin{equation}
G\left( z,\overset{\cdot }{z}\right) =\Pi _{2}\left\{ \Gamma \left( \Psi
\left( \Gamma ^{-1}Y\right) \right) \right\} +O\left( \left\Vert
Y\right\Vert ^{R+1}\right) \text{ .}  \label{G(z zdot)}
\end{equation}%
Let us now return to Eq.\ref{second order general}. We shall look for an
oscillation with positive and small, but finite, amplitude $\varepsilon $.
The small parameter $\varepsilon $ is connected with the small bifurcation
parameter $\tau _{a}$ and will be defined a bit later. Taking in Eq.\ref%
{second order general} the change of variables%
\begin{equation}
z\left( t\right) =\varepsilon \varsigma \left( t\right)  \label{zeta-zita}
\end{equation}%
we get the equation of a weakly nonlinear oscillator in \emph{normal form}: 
\begin{equation}
\overset{\cdot \cdot }{\varsigma }-\tau _{a}\overset{\cdot }{\varsigma }+\
\delta _{a}\ \varsigma =\ \varepsilon \ G\left( \varsigma ,\overset{\cdot }{%
\varsigma };\varepsilon \right) \text{ .}  \label{eq for zita}
\end{equation}%
Then, to each periodic solution to Eq.\ref{linearization} will correspond a
non-trivial periodic solution to Eq.\ref{eq for zita}. In \cite{mrrsm}, was
considered the Krylov-Bogoliubov averaging method \cite{bogoliubov} to
derive an asymptotic expansion to the solution of Eq.\ref{eq for zita}. To
do so, let us consider the new variables $r=r\left( t\right) $ and $\theta
=\theta \left( t\right) $ defined as follows%
\begin{eqnarray}
\varsigma &=&r\cos \left( t+\theta \right)  \label{polar zita} \\
\overset{\cdot }{\varsigma } &=&-r\sin \left( t+\theta \right)
\label{polar zita dot}
\end{eqnarray}%
then, the corresponding averaged equations are%
\begin{equation}
\overset{\cdot }{r}=-\frac{1}{2\pi }\dint\limits_{0}^{2\pi }\sin \phi \
\left\{ -\tau _{a}\ r\sin \phi +\varepsilon G\left( r\cos \phi ,-r\sin \phi
;\varepsilon \right) \right\} \ d\phi  \label{averaging r}
\end{equation}%
\begin{equation}
\overset{\cdot }{\theta }=-\frac{1}{2\pi r}\dint\limits_{0}^{2\pi }\cos \phi
\ \left\{ -\tau _{a}\ r\sin \phi +\varepsilon G\left( r\cos \phi ,-r\sin
\phi ;\varepsilon \right) \right\} \ d\phi  \label{averaging theta}
\end{equation}%
thus,%
\begin{equation}
\overset{\cdot }{r}=\frac{r}{2}\left\{ \tau _{a}-p\left( r;\varepsilon
\right) \right\}  \label{r general}
\end{equation}%
\begin{equation}
\overset{\cdot }{\theta }=q\left( r;\varepsilon \right)
\label{theta general}
\end{equation}%
where%
\begin{equation}
p\left( r;\varepsilon \right) =\frac{\varepsilon }{\pi r}\dint\limits_{0}^{2%
\pi }\sin \phi \ G\left( r\cos \phi ,-r\sin \phi ;\varepsilon \right) \ d\phi
\label{p def}
\end{equation}%
\begin{equation}
q\left( r;\varepsilon \right) =-\frac{\varepsilon }{2\pi r}%
\dint\limits_{0}^{2\pi }\cos \phi \ G\left( r\cos \phi ,-r\sin \phi
;\varepsilon \right) \ d\phi \text{ .}  \label{q def}
\end{equation}

Let us quote now some important properties about the functions $p$ and $q$
above.

\begin{proposition}
\label{prop oddpowers}Functions \bigskip $p\left( r;\varepsilon \right) $
and $q\left( r;\varepsilon \right) $ have at least order $O\left(
\varepsilon ^{2}\right) $ or, equivalently, $p\left( r;\varepsilon \right)
/r^{2}$ and $q\left( r;\varepsilon \right) /r^{2}$ have a finite limit as $%
r\rightarrow 0$. Moreover, the Taylor expansions of $p\left( r;\varepsilon
\right) $ and $q\left( r;\varepsilon \right) $ must not contain odd powers
of $r$.
\end{proposition}

From Proposition \ref{prop oddpowers} it can be concluded that the
development of $p$ has the form%
\begin{equation}
p\left( r;\varepsilon \right) =p_{3}\ \varepsilon ^{2}r^{2}+p_{5}\
\varepsilon ^{4}r^{4}+\cdots  \label{p develop}
\end{equation}%
in which $p_{s}=p_{s}\left( \tau _{a}\right) $, as a consequence of the
dependence that the right hand side of Eq.\ref{variations} has on the
bifurcation parameter $\tau _{a}$. In Subsections \ref{subsection multiple}, %
\ref{subsection multiple II} and \ref{Subsect deg 1stkind} can be found some
examples showing this type of dependence.

We recall now the fact that, in accordance with Eqs.\ref{r general} and \ref%
{p develop}, the $M$ coefficients $\left( p_{2j+1}\right) _{j=1}^{M}$ will
determine the bifurcation up to $2\left( M+1\right) $-jet-equivalence. So,
it can be expected that we would manage the coefficients $p_{2j+1}$ by
taking $M$ appropriate independent relations involving the parameters in the
system Eq.\ref{variations}. This scenario is called a codimension-$M$
bifurcation. In accordance with our purpose in this paper, let us now
introduce a definition that will be essential for the classification of HB.
The reason to take it, will arise in the next Section.

\begin{remark}
\label{smooth}We also recall that we are considering smooth vector fields,
while in the construction of the analytical transform of coordinates in Eq.%
\ref{Y=Gamma X} we use analytical properties. But notice that the transform
Eq.\ref{Y=Gamma X} as well as the appropriate $R$-jet approximation, which
is analytical, lead to functions Eq.\ref{G(z zdot)} in which the principal
part remains unaltered after their substitutions into the smooth system Eq.%
\ref{variations}.
\end{remark}

\begin{definition}
\label{remark neglected}Let $p_{2s+1}$ be a coefficient in the formal
development Eq.\emph{\ref{p develop}, }which is derived from the formal $%
\infty $-jet of $F$. It\emph{\ }shall be called \emph{negligible} if
satisfies 
\begin{equation}
\left\vert p_{2s+1}\right\vert \leq K_{s}\left\vert \tau _{a}\right\vert
\label{p2s+1 neglec}
\end{equation}%
for certain constant $K_{s}>0$ as $\tau _{a}\rightarrow 0$. The function $%
p\left( r;\varepsilon \right) $ in Eq.\ref{p develop} is said to be \emph{%
negligible} if for all $s\in 
\mathbb{N}
$ the coefficient $p_{2s+1}$ is negligible.
\end{definition}

Naturally, $p_{2s+1}\equiv 0$ and $p\equiv 0$ are respectively included in
the above definition. For instance, $p\equiv 0$ if in the formal $j_{\infty
}\left( F\right) \left( 0\right) $ the non-vanishing terms have even degree.
Moreover, we can get $p_{2s+1}\equiv 0$ in spite of the existence of
non-zero coefficients with degree $2s+1$ in the formal Taylor development of
the right hand in Eq.\ref{variations}. As we shall see in the next Section,
negligible terms have no influence in the generation of limit cycles.
Besides, with this notion we are able to give a more detailed version of the
Proposition 3 in \cite{mrrsm} as follows.

\begin{proposition}
If the function $p\left( r;\varepsilon \right) $ is non-negligible, there
must exist a positive integer $N$ and a positive real value $%
r_{0}=r_{0}\left( \tau _{a}\right) $ such that $p\left( r,\varepsilon
\right) $ has the non-trivial Taylor expansion: 
\begin{equation}
p\left( r;\varepsilon \right) =\omega \ \varepsilon ^{2N}\ r_{0}^{-2N}\
r^{2N}+O\left( \varepsilon ^{2N+2}\ r^{2N+2}\right)  \label{p}
\end{equation}%
where $\omega =+1$ or $-1$. In addition, the behavior of the factor $%
r_{0}^{-2N}$ as $\tau _{a}\rightarrow 0$ obeys the following alternative:
either 
\begin{equation}
\underset{\tau _{a}\rightarrow 0}{\lim }r_{0}^{-2N}=r_{\ast }^{-2N}>0
\label{lim ro nonzero}
\end{equation}%
or, for a given $\gamma $, $0<\gamma <1$,%
\begin{equation}
r_{0}^{-2N}=O_{S}\left( \left\vert \tau _{a}\right\vert ^{\gamma }\right) 
\text{ as }\tau _{a}\rightarrow 0\text{ .}  \label{degen2dkind}
\end{equation}
\end{proposition}

As in \cite{sanders}, the symbol $O_{S}$ in Eq.\ref{degen2dkind} means a
sharp estimate, that is: $r_{0}^{-2N}=O\left( \left\vert \tau
_{a}\right\vert ^{\gamma }\right) $ and $r_{0}^{-2N}\neq o\left( \left\vert
\tau _{a}\right\vert ^{\gamma }\right) $ as $\tau _{a}\rightarrow 0$. We
recall that the bifurcation is supercritical (respect. subcritical) if $%
\omega =+1$ (respect. $\omega =-1$). In the supercritical case, the root $%
r_{0}$ appears for $\tau _{a}>0$ (respect. $\tau _{a}<0$), so the limit in
Eq.\ref{lim ro nonzero} or the order relation in Eq. \ref{degen2dkind}
should be considered as $\tau _{a}\rightarrow 0+$ (respect. $\tau
_{a}\rightarrow 0-$). Consequently,

\begin{proposition}
\label{prop limit cycle existence}Let us assume that $p\left( r;\varepsilon
\right) $ be non-negligible and also, that $r_{0}$ in Eq.\ref{p} has the
property in Eq.\ref{lim ro nonzero}. Hence, there is a positive root $\rho $
to the equation 
\begin{equation}
p\left( r;\varepsilon \right) -\tau _{a}=0  \label{limit cycle existence}
\end{equation}%
either for positive or for negative values $\tau _{a}$ close enough to zero.
Furthermore, up to the leading term, the root to Eq.\ref{limit cycle
existence} has the form 
\begin{equation}
\rho =\left( \frac{\left\vert \tau _{a}\right\vert }{\varepsilon ^{2N}}%
\right) ^{1/2N}\left( r_{\ast }+O\left( \left\vert \tau _{a}\right\vert
\right) \right) \ +O\left( \varepsilon ^{2}\right) \text{ .}  \label{root}
\end{equation}
\end{proposition}

\begin{proof}
The property Eq.\ref{lim ro nonzero} is equivalent to $r_{0}^{-2N}=r_{\ast
}^{-2N}+O\left( \left\vert \tau _{a}\right\vert \right) $ and, also to $%
r_{0}=r_{\ast }+O\left( \left\vert \tau _{a}\right\vert \right) $.
\end{proof}

Due to Proposition \ref{prop limit cycle existence} the function $p$ in Eq.%
\ref{p def}\ has been called a \emph{discriminant}\ for the HB in \cite%
{mrrsm}. Let us assume the existence of a (finite) positive root Eq.\ref%
{root} such that Eq.\ref{lim ro nonzero} holds. Then, the small parameter $%
\varepsilon $ introduced in Eq.\ref{zeta-zita} can be taken as%
\begin{equation}
\varepsilon ^{2N}=\left\vert \tau _{a}\right\vert \text{\ .}  \label{epsilon}
\end{equation}%
From Eq.\ref{epsilon} and the relation $\left\vert \tau _{a}\right\vert
=O\left( \left\vert \tau _{a}\right\vert ^{1/N}\right) $ as $\tau
_{a}\rightarrow 0$ for $N\geq 1$, it follows that Eq.\ref{root} can now be
written as%
\begin{equation}
\rho =r_{\ast }+\ O\left( \left\vert \tau _{a}\right\vert ^{1/N}\right) 
\text{ .}  \label{root ro}
\end{equation}%
If Eq.\ref{lim ro nonzero} does not hold we may assume Eq.\ref{degen2dkind},
so in a similar way as we do in Proposition \ref{prop limit cycle existence}
to get Eq.\ref{epsilon}, we now arrive to%
\begin{equation}
\varepsilon ^{2N}=\left\vert \tau _{a}\right\vert ^{1-\gamma }\text{ .}
\label{epsilon deg}
\end{equation}%
Moreover, if for instance 
\begin{equation*}
r_{0}^{-2N}=r_{L}\left\vert \tau _{a}\right\vert ^{\gamma }+o\left(
\left\vert \tau _{a}\right\vert ^{\gamma }\right) \text{ as }\tau
_{a}\rightarrow 0\text{ }
\end{equation*}%
for certain positive number $r_{L}$, then Eq.\ref{root} can be rewritten as%
\begin{equation*}
\rho =\left\vert \tau _{a}\right\vert ^{\left( 1-\gamma \right) /2N}\ \frac{%
r_{L}}{\varepsilon }\ +O\left( \varepsilon ^{2}\right)
\end{equation*}%
and follows 
\begin{equation}
\rho =r_{L}+\ O\left( \left\vert \tau _{a}\right\vert ^{\left( 1-\gamma
\right) /N}\right) \text{ .}  \label{root rL}
\end{equation}

The hypothesis of the HB Theorem, as appear in \cite[Theorem 1]{mrrsm},
contains an implicit reference to the non-degenerate case, more precisely,
to the condition Eq.\ref{lim ro nonzero}. As we shall see in Section \ref%
{Section classification}, Eqs.\ref{lim ro nonzero}, \ref{degen2dkind} and %
\ref{p2s+1 neglec} put in evidence the reason for differentiation in HB we
are suggesting here. In accordance with the intention and terminology in
this paper, we shall give in Section \ref{Section classification} a split
version of the HB theorem.

\section{Theorems for degenerate Hopf bifurcation \label{Section
classification}}

Basically, the standard classification of HB is conformed by two main
classes: while the system at degenerate HB shows a center at the critical
value $a_{\ast }$, the system at a non-degenerate HB shows a weak focus at
the critical value, leading to the emergence (super- or subcritical) of a
limit cycle \cite{edelstein}. When the bifurcation occurs in a one-parameter
family of vector fields whose first non-zero derivatives at the origin have
order $2N+1$, $N>1$, it is called \cite{gasull} a generalization of the
Andronov-Hopf's. Other higher codimension HB, for instance the Bautin
bifurcation, are often called generalized \cite{kuznetsov}. In our
formulation, we shall include such higher codimension HB, hence we shall
assume we have a family of systems parametrized by the bifurcation parameter 
$\tau _{a}$. Consequently, the coefficients in Eq.\ref{p develop} also
depend on the bifurcation parameter, $p_{s}=p_{s}\left( \tau _{a}\right) $,
so we can \textit{a priori} classify the HB taking into account how these
dependences are.

\begin{definition}
\label{definition bifurcations}We shall say that the Hopf bifurcation is 
\emph{degenerate of first kind} if $p$ is negligible (See Definition \ref%
{remark neglected} ). Let $N$ be given as in Eq.\ref{p}.\emph{\ }The
bifurcation shall be called \emph{degenerate of second kind}, if there
exists a number $\gamma $, $0<\gamma <1$, such that Eq.\ref{degen2dkind}
holds. The HB shall be called \emph{non-degenerate}, provided Eq.\ref{lim ro
nonzero}.
\end{definition}

For instance, if $p\equiv 0$ in Eq.\ref{p develop}, the bifurcation shall be
degenerate of first kind. Further, the existence of at least one
non-negligible $p_{2s+1}$ derived from the formal $\infty $-jet of $F$, no
matter how large the number $s$ is, implies that the HB will not be
degenerate of first kind. We also remark that, at degenerate bifurcation, we
are implicitly considering that the system moves close to a higher
codimension point as the bifurcation parameter varies, because the non-zero
coefficients $p_{2s+1}$vanish at $\tau _{a}=0.$

\begin{remark}
There is a reason to take $0<\gamma <1$ in the definition of the degeneracy
of second kind above. As $F$ in Eq.\ref{variations} is a continuous function
of $\tau _{a}$ in a vicinity of $\tau _{a}=0$, we may assume $\gamma >0$ as
a consequence of Eq.\ref{degen2dkind}. Furthermore, from Eq.\ref{epsilon deg}
and the \textquotedblleft smallness\textquotedblright\ of the parameter $%
\varepsilon $ in Eq.\ref{zeta-zita} as $\tau _{a}\rightarrow 0$ follows that 
$1-\gamma >0$.
\end{remark}

It is easy to see, from Definition \ref{definition bifurcations}, that
degenerate (of any kind) HB implicitly implies that the system shows a
center at the critical value of the bifurcation parameter. This is because
the main terms of $p$ in Eq.\ref{r general} tends to zero as $\tau
_{a}\rightarrow 0$, so the origin is a center at the critical value. We
split up the HB theorem taking into account each type of bifurcation. We
recall that the bifurcation is supercritical if the cycle emerges provided $%
0<\tau _{a}\ll 1$, and subcritical if it emerges for $0<-\tau _{a}\ll 1$.

\begin{theorem}[non-degenerate Hopf bifurcation]
\label{hopf th}\bigskip Let us assume that Eq.\ref{complex eigenvalues cond}
holds and that Eq.\ref{limit cycle existence} has a root Eq.\ref{root ro}
with the property Eq.\ref{lim ro nonzero} for positive (respectively,
negative) but sufficiently close to zero values of the bifurcation parameter 
$\tau _{a}$. Then, a single limit cycle to the system in Eq.\ref%
{linearization} emerges. Furthermore, the limit cycle is orbitally
asymptotically stable (respect., unstable) if and only if the bifurcation is
supercritical (respect., subcritical). The radius of the emerging cycle is $%
r=O_{S}\left( \left\vert \tau _{a}\right\vert ^{1/2N}\right) $, while the
frequency is $\varpi =1+O\left( \left\vert \tau _{a}\right\vert
^{1/N}\right) $ as $\tau _{a}\rightarrow 0$.
\end{theorem}

\begin{proof}
The proof follows from Eq.\ref{limit cycle existence} and Eq.\ref{r general}%
. Let the main term in Eq.\ref{p} have the property Eq.\ref{lim ro nonzero}.
Then, a single root to Eq.\ref{limit cycle existence} tends to zero as $\tau
_{a}\rightarrow 0$. This fact means that a single limit cycle emerges at
bifurcation. It is not excluded the existence of further different roots to
Eq.\ref{limit cycle existence}, but if any other appears, it determines a
limit cycle that does not vanish in a vicinity of $\tau _{a}=0$, so the
cycle \textquotedblleft persists\textquotedblright\ along the bifurcation.
As follows from Eqs.\ref{ubar gen}, \ref{vbar gen} and \ref{v1 (t)}, we get
the order of the radius of the limit cycle. The order of the frequency is
determined in Subsection \ref{Subsection frequency}.
\end{proof}

Consequently, a necessary but not sufficient condition for the emergence of
a limit cycle at bifurcation, is the existence of nonzero odd order terms in
the expansion of the components in the right-hand side of Eq.\ref{variations}%
. For instance, the following result states that the condition is not
sufficient.

\begin{theorem}[degenerate HB of first kind]
\label{hopf degenerate 1st}Let us assume that Eq.\ref{complex eigenvalues
cond} holds and assume $p$ negligible (Definition \ref{remark neglected}).
Then, neither limit cycle surrounding the steady state emerges after nor
before the critical value.
\end{theorem}

\begin{proof}
From Eq.\ref{p2s+1 neglec} and Eq.\ref{limit cycle existence} follows that
none of the roots to Eq.\ref{limit cycle existence} tends to zero as $\tau
_{a}\rightarrow 0$. More precisely, if any root exists, it tends to
infinity. Taking into account only the leading terms we may write%
\begin{equation*}
p_{2k+1}=\alpha _{k}\ \tau _{a}^{1+\mu _{k}}
\end{equation*}%
where $\mu _{k}\geq 0$, $\alpha _{k}$ are constants which can be zero. Let $%
N $ corresponds to the first non-zero $\alpha _{k}$, and take $\varepsilon
^{2N}=$ $\tau _{a}$ (if the bifurcation is supercritical), so Eq.\ref{p
develop} is 
\begin{equation*}
p=\varepsilon ^{2N}\sum_{n=N}^{\infty }\alpha _{k}\ \tau _{a}^{1+\mu
_{k}}\varepsilon ^{2\left( n-N\right) }r^{2N}=\tau
_{a}^{2}\sum_{n=N}^{\infty }\alpha _{k}\ \tau _{a}^{\mu _{k}+\left(
n-N\right) /N}r^{2N}
\end{equation*}%
hence, if there is a root to $\tau _{a}-p\left( r;\varepsilon \right) =0$
then it should tend to infinity as $\tau _{a}\rightarrow 0$.
\end{proof}

\begin{remark}
We remark that, at a degenerate HB of first kind, further limit cycles may
persist in a neighborhood of the critical value of the bifurcation parameter
as can be seen in Subsection \ref{Subsect deg 1stkind}. Furthermore, if the
order of a negligible term $p_{2s+1}$ is greater than one, then limit cycles
surrounding the point at infinity may emerge in the so-called \emph{HB at
the infinity} (see \cite{gasull} and references therein). In Subsection \ref%
{Subsect deg 1stkind} we shall show an example of this situation. For
instance, this is the case if 
\begin{equation*}
p_{2s+1}=A\ \tau _{a}^{\beta +1}+o\left( \left\vert \tau _{a}\right\vert
^{\beta +1}\right)
\end{equation*}%
being $\tau _{a}\rightarrow 0$, $\beta >0$, $A\neq 0$.
\end{remark}

Now, with the same procedure in the proof of Theorem \ref{hopf th}, we have,

\begin{theorem}[degenerate HB of second kind]
\label{hopf degenerate th}Let us assume that Eq.\ref{complex eigenvalues
cond} holds and that Eq.\ref{limit cycle existence} has a root Eq.\ref{root
ro} with the property Eq.\ref{degen2dkind} for positive (respectively,
negative) but sufficiently close to zero values of the parameter $\tau _{a}$%
. Then, it can be assured the emergence of at least one limit cycle to the
system in Eq.\ref{variations}, the radius of which has order 
\begin{equation}
r=O_{S}\left( \left\vert \tau _{a}\right\vert ^{\left( 1-\gamma \right)
/2N}\right)  \label{radius degenerate}
\end{equation}%
while the frequency is $\varpi =1+O\left( \left\vert \tau _{a}\right\vert
^{\left( 1-\gamma \right) /N}\right) $ as $\tau _{a}\rightarrow 0$. The
number $\gamma $, $0<\gamma <1$ corresponds to the one in Eq.\ref%
{degen2dkind}.
\end{theorem}

\begin{proof}
The proof for the non-degenerate bifurcation can be repeated, but taking the
small parameter from Eq.\ref{epsilon deg}. Notice that in this case more
than a single positive root to Eq.\ref{limit cycle existence} tending to
zero as $\tau _{a}\rightarrow 0$, may appear.
\end{proof}

At a degenerate HB of second kind, different behaviors can be observed: a
single or several limit cycles may emerge, including semi-stable cycles as
the limiting case at which different limit cycles collapse. In Section \ref%
{SectionNormalForms} some examples with this kind of degeneracy are
considered to show that, multiple limit cycles or semistable limit cycles
might appear either sub- or supercritically. Bearing in mind the
construction of the examples in Section \ref{SectionNormalForms}, looks easy
to find sufficient conditions for the emergence of multiple limit cycles, as
it is done in Subsection \ref{subsection sufficient degenerate}.

We can show sufficient conditions for the stability of the emerging limit
cycles surrounding the steady state. Such conditions are based on the
behavior of the discriminant $p$ near the root to which the cycle
corresponds.

\begin{proposition}
Consider a root $\rho $ to Eq.\ref{limit cycle existence}, so it corresponds
to a limit cycle $L$. This cycle is asymptotically stable or unstable if the
number $dp/dr\left( \rho \right) $ is negative or positive respectively.
\end{proposition}

\begin{proof}
In the calculation of $dp/dr\left( \rho \right) $ we can assume $\rho =r_{0}$
as in Eq.\ref{root ro}, or $\rho =r_{L}$ in Eq.\ref{root rL}, in accordance
with the type of bifurcation. Due to the continuity argument near $\tau
_{a}=0$ in Eqs. \ref{root ro} or \ref{root rL}, and the fact that $%
dp/dr\left( r_{0}\right) $ (or $dp/dr\left( r_{L}\right) $) does not vanish,
the assertion follows.
\end{proof}

\subsection{The Hopf coefficients \label{section HB}}

Let us take $M\geq 2$ in Eq.\ref{variations} and the map Eq.\ref{Y=Gamma X}
in which $\mathcal{H}$ has a contact with\emph{\ }\textbf{$\Gamma $} of
order $M$ at the origin, we get%
\begin{equation*}
\overset{\cdot }{Y}=\dsum\limits_{1\leq k+l\leq M}\Gamma \left( 
\begin{array}{c}
\sigma _{kl}^{1} \\ 
\sigma _{kl}^{2}%
\end{array}%
\right) \left( \mu _{11}z+\mu _{12}\overset{\cdot }{z}\right) ^{k}\ \left(
\mu _{21}z+\mu _{22}\overset{\cdot }{z}\right) ^{l}+O\left( \left(
\left\Vert Y\right\Vert ^{M+1}\right) \right)
\end{equation*}%
where $\Gamma =\left( \gamma _{ij}\right) $ and $\Gamma ^{-1}=\left( \mu
_{ij}\right) $. For instance, if we take $\Gamma =\Gamma _{1}$, the
quantities 
\begin{equation}
\mu _{11}=1\text{ ; }\mu _{12}=0\text{ ; }\mu _{21}=\tau _{a}\Lambda
_{a}^{-1}\text{ ; }\mu _{22}=-2\Lambda _{a}^{-1}  \label{MUij}
\end{equation}%
are the components of $\Gamma _{1}^{-1}$. So, Eq.\ref{G} can be written%
\begin{eqnarray}
G\left( z,\overset{\cdot }{z}\right) &=&\dsum\limits_{2\leq k+l\leq
M}R_{kl}\left( \mu _{11}z+\mu _{12}\overset{\cdot }{z}\right) ^{k}\ \left(
\mu _{21}z+\mu _{22}\overset{\cdot }{z}\right) ^{l}  \label{G explicit} \\
&&+O\left( \left( \left\Vert Y\right\Vert ^{M+1}\right) \right)  \notag
\end{eqnarray}%
where%
\begin{equation}
R_{kl}=\gamma _{21}\ \sigma _{kl}^{1}\ +\gamma _{22}\ \sigma _{kl}^{2}\text{
.}  \label{Rkl}
\end{equation}

Hence, Eq.\ref{G explicit} can be rewritten in terms of powers of the $Y$%
-components as%
\begin{equation}
G\left( z,\overset{\cdot }{z}\right) =\dsum\limits_{2\leq k+l\leq M}H_{kl}\
z^{k}\ \left( \overset{\cdot }{z}\right) ^{l}+O\left( \left( \left\Vert
Y\right\Vert ^{M+1}\right) \right) \text{ .}  \label{G Hopfcomb}
\end{equation}%
Only the non-zero $H_{mn}$ in Eq.\ref{G Hopfcomb} corresponding to such
pairs $\left( m,n\right) $ for which $K_{mn}\neq 0$, where 
\begin{equation}
K_{mn}=\int_{0}^{2\pi }\cos ^{m}\phi \ \sin ^{n+1}\phi \ d\phi
\label{integral(m,n)}
\end{equation}%
can contribute to the appearance of a nonzero term in Eq.\ref{p} so, to the
appearance of a limit cycle solution.

Let us introduce the following

\begin{definition}
\label{def hopf coeff}We shall call the \emph{Hopf coefficient} of degree $%
\left( 2N+1\right) $ to the coefficient $p_{2N+1}$ in the expansion Eq.\ref%
{p develop}, which is an algebraic combination of coefficients $H_{mn}$ in
Eq.\ref{G Hopfcomb} provided $m+n=2N+1$.
\end{definition}

From Eq.\ref{polar zita dot}, Eq.\ref{G Hopfcomb} and Eq.\ref{p def} follows
directly that 
\begin{equation}
p_{3}=-\frac{1}{4}\left( 3H_{03}+H_{21}\right)  \label{HC3}
\end{equation}%
is the Hopf coefficient of third degree, and%
\begin{equation}
p_{5}=-\frac{1}{8}\left( 5H_{05}+H_{23}+H_{41}\right)  \label{HC5}
\end{equation}%
is the Hopf coefficient of degree five. Other Hopf coefficients can be
derived by forward calculations. We recall that are required only the
coefficients of $\Gamma $ and the coefficients of the main part of the
Taylor expansion, to calculate the Hopf coefficients. Becomes easy to check
from Eq.\ref{integral(m,n)} (see Appendix \ref{section appendix}) that the
coefficients in Eq.\ref{G Hopfcomb} leading to Hopf coefficients in Eq.\ref%
{G(z zdot)} up to $6$-jet equivalence are%
\begin{eqnarray}
&&%
\begin{array}{cc}
H_{0,3} & H_{2,1}%
\end{array}%
\text{ }  \label{hopf terms} \\
&&%
\begin{array}{ccc}
H_{0,5} & H_{2,3} & H_{4,1}%
\end{array}%
\text{ .}  \notag
\end{eqnarray}

In accordance with Eq.\ref{Rkl}, if the Jacobian matrix have the form Eq.\ref%
{J} and being $\Gamma =\Gamma _{1}$, we have the numbers%
\begin{equation}
R_{mn}=\frac{1}{2}\left( \tau _{a}\ \sigma _{m,n}^{1}-\Lambda _{a}\ \sigma
_{m,n}^{2}\right)  \label{R mn}
\end{equation}%
We recall that the matrix $\Gamma $, as a linear combination of the matrixes
in Eq.\ref{gamma generators}, depends on the elements of the matrix $J_{a}$
which is supposed to have the form in Eq.\ref{J}. Then, taking $\Gamma
=\Gamma _{1}$ and $\Gamma ^{-1}=\left( \mu _{ij}\right) $, we get the
following equalities 
\begin{equation}
H_{03}\underset{\mathtt{def}}{=}-\left( 8\ R_{03}\right) \Lambda _{a}^{-3}%
\text{ ,}  \label{H 03}
\end{equation}%
\begin{equation}
H_{21}\underset{\mathtt{def}}{=}-2\left( R_{21}\ \Lambda _{a}^{2}+2R_{12}\
\tau _{a}\Lambda _{a}+3R_{03}\ \tau _{a}^{2}\right) \Lambda _{a}^{-3}\text{ ,%
}  \label{H 21}
\end{equation}%
\begin{equation}
H_{05}\underset{\mathtt{def}}{=}-\left( 32\ R_{05}\right) \left( \Lambda
_{a}^{-5}\right) \text{ ,}  \label{H 05}
\end{equation}%
\begin{equation}
H_{23}\underset{\mathtt{def}}{=}-8\left( R_{23}\ \Lambda _{a}^{2}+4R_{14}\
\tau _{a}\Lambda _{a}^{{}}+10R_{05}\ \tau _{a}^{2}\Lambda _{a}^{2}\right)
\Lambda _{a}^{-5}\text{ ,}  \label{H 23}
\end{equation}%
\begin{eqnarray}
&&H_{41}\underset{\mathtt{def}}{=}-2\left( R_{41}\ \Lambda _{a}^{4}+2R_{32}\
\tau _{a}\Lambda _{a}^{3}+3R_{23}\ \tau _{a}^{2}\Lambda _{a}^{2}\right.
\label{H 41} \\
&&\left. \ \ \ \ \ \ \ \ \ \ \ +4R_{14}\ \tau _{a}^{3}\Lambda
_{a}^{{}}+5R_{05}\ \tau _{a}^{4}\right) \Lambda _{a}^{-5}\text{ ,}  \notag
\end{eqnarray}%
from Eq.\ref{G Hopfcomb}.

It is concluded in this Subsection that, up to $6$-jet-equivalence, we can
easily check the type of degeneracy the degenerate HB shows. Particularly,
Eqs.\ref{HC3} and \ref{HC5} give the full information for polynomial vector
fields of degree at most six. In Section \ref{SectionNormalForms} we show
different examples of degenerate HB within the class of polynomial vector
fields of degree at most six.

\subsection{Asymptotic expansions of the limit cycles}

Going back to the substitutions given in Eq.\ref{zeta-zita} and Eq.\ref%
{Y=Gamma X}, we could derive uniform asymptotic expansion to the solution of
Eq.\ref{variations}. So, if the HB is non-degenerate, it is possible to
develop the periodic solution $\Theta \left( t\right) =\left( \overline{x}%
\left( t\right) ,\overline{y}\left( t\right) \right) $ to Eq.\ref{variations}
generating the limit cycle, by 
\begin{equation}
\overline{x}\left( t\right) =\ u_{1}\left( t\right) \ \left( \left\vert \tau
_{a}\right\vert \right) ^{\frac{1}{2N}}+O\left( \left\vert \tau
_{a}\right\vert ^{\frac{1}{N}}\right)  \label{ubar gen}
\end{equation}%
\begin{equation}
\overline{y}\left( t\right) =v_{1}\left( t\right) \ \left( \left\vert \tau
_{a}\right\vert \right) ^{\frac{1}{2N}}\ +O\left( \left\vert \tau
_{a}\right\vert ^{\frac{1}{N}}\right)  \label{vbar gen}
\end{equation}%
and%
\begin{equation}
\left( 
\begin{array}{c}
u_{1}\left( t\right) \\ 
v_{1}\left( t\right)%
\end{array}%
\right) =r_{\ast }\ \Gamma ^{-1}\ \left( 
\begin{array}{c}
\cos \left( \varpi \ t\right) \\ 
-\sin \left( \varpi \ t\right)%
\end{array}%
\right)  \label{v1 (t)}
\end{equation}%
with frequency 
\begin{equation}
\varpi =1+q\left( \rho ,\left\vert \tau _{a}\right\vert ^{\frac{1}{2N}%
}\right)  \label{frequencyLC}
\end{equation}%
and period $T=\frac{2\pi }{\varpi }$. From Eq.\ref{theta general} the
angular speed of the oscillation is obtained. Note that, up to the leading
terms, the expansions of the cycle solution in Eqs.\ref{ubar gen} and \ref%
{vbar gen} are uniform, as the $O\left( \left\vert \tau _{a}\right\vert ^{%
\frac{1}{2N}}\right) $-terms are bounded functions. The expansions in Eqs.%
\ref{ubar gen} and \ref{vbar gen}, can be taken also for the limit cycle
that emerges at degenerate HB of second kind considered in Theorem \ref{hopf
degenerate th}, but substituting the exponent $\frac{1}{2N}$ by $\frac{%
1-\gamma }{2N}$, and $r_{\ast }$ by the $r_{L}$ given in Eq.\ref{root rL}.

\subsection{On the period of the limit cycles\label{Subsection frequency}}

Here we give a brief comment about the period of the limit cycles,
considering the interest in the topic \cite{gasull}. The period of the
emerging limit cycles in non-degenerate or degenerate of second kind HB can
be determined from the formula for the frequency, given in Eq.\ref%
{frequencyLC}. To do so, it is necessary to consider Eqs.\ref{averaging
theta}, \ref{q def} and \ref{G Hopfcomb}. Up to $6$-jet-equivalence, the
expansion results%
\begin{equation*}
q=-\frac{1}{8}\varepsilon ^{2}\rho ^{2}\left( \left[ 3H_{30}+H_{12}\right] +%
\frac{1}{2}\varepsilon ^{2}\rho ^{2}\left[ 5H_{50}+H_{32}+H_{14}\right]
\right) +O\left( \varepsilon ^{6}\right) \text{ .}
\end{equation*}%
The coefficients $H_{kl}$ in the above formula, can be calculated in a
similar way as the others in Eqs.\ref{H 03} to \ref{H 41}. If the HB is
non-degenerate we have the relation Eq.\ref{epsilon}, while we have to
consider Eq.\ref{epsilon deg} if the bifurcation is degenerate of second
kind. Further, the factor $\rho $ take the values Eq.\ref{root ro} or Eq.\ref%
{root rL} in accordance with the type of bifurcation. The procedure yields
to the following estimates of the frequency as $\tau _{a}\rightarrow 0$: 
\begin{equation}
q=O\left( \left\vert \tau _{a}\right\vert ^{\frac{1}{N}}\right)
\label{frequencyND}
\end{equation}%
for the emerging cycle in the non-degenerate case, and 
\begin{equation}
q=O\left( \left\vert \tau _{a}\right\vert ^{\frac{1-\gamma }{N}}\right)
\label{frequencyDeg}
\end{equation}%
if the bifurcation is degenerate of second kind. We remark the fact that, in
the last two formulas, the order it is not necessarily sharp. In the
examples in the next Section we also have done a reference about the period.

\section{Typical forms in degenerate HB \label{SectionNormalForms}}

The main concern in this Section are degenerate HB of first and second
kinds. In Subsections \ref{subsection multiple} to \ref{cycles at infinity}
we shall study \textquotedblleft typical\textquotedblright\ forms by using
polar coordinates in order to give a familiar, more geometrical, description
of the nature of the bifurcation. Notice that the examples in the referenced
Subsections are not normal forms, because in these cases no genericity
conditions (see \cite{kuznetsov}) are involved. It is not difficult to show
that, using the averaging method we can get the same conclusions. In
Subsection \ref{subsection sufficient degenerate} we state, as an example of
what can be expected in presence of a degenerate HB of second kind, a
sufficient condition using the Hopf coefficients up to $6$-jet-equivalence.

\subsection{Multiple cycles in supercritical degenerate Hopf bifurcation of
second kind\label{subsection multiple}}

In this paragraph we shall give an example of a system which shows a
degenerate HB of second kind in accordance with Definition \ref{definition
bifurcations}:

\begin{equation}
\left( 
\begin{array}{c}
\overset{\cdot }{x} \\ 
\overset{\cdot }{y}%
\end{array}%
\right) =\left( 
\begin{array}{cc}
a^{3} & -1 \\ 
1 & a^{3}%
\end{array}%
\right) \left( 
\begin{array}{c}
x \\ 
y%
\end{array}%
\right) -a^{2}\left( x^{2}+y^{2}\right) \left( 
\begin{array}{c}
x \\ 
y%
\end{array}%
\right) +\frac{3}{16}a\left( x^{2}+y^{2}\right) ^{2}\left( 
\begin{array}{c}
x \\ 
y%
\end{array}%
\right) \text{.}  \label{An HB}
\end{equation}%
Neither limit cycle surrounds the focus while the trace $\tau =2a^{3}$ is
negative. At the critical value $a=0$, the origin is clearly a center to the
system Eq.\ref{An HB}. Further, when the trace becomes positive the steady
state becomes unstable and two small limit cycles emerge with radius%
\begin{equation*}
r_{1}=\frac{2\sqrt{3}}{3}a^{1/2}\text{ \ and\ \ }r_{2}=2a^{1/2}
\end{equation*}%
respectively. Both cycles emerge due to the bifurcation and have radius $%
O_{S}\left( \tau ^{1/6}\right) $ as $\tau \rightarrow 0$. The first cycle is
stable, while the second is unstable. To check the above assertions we only
need to rewrite the system Eq.\ref{An HB} in polar coordinates. We get the
system:%
\begin{equation*}
\left\{ 
\begin{array}{c}
\overset{\cdot }{r}=ar\left( a^{2}-ar^{2}+\frac{3}{16}r^{4}\right) \\ 
\overset{\cdot }{\theta }=1\text{ }%
\end{array}%
\right. \text{ .}
\end{equation*}%
From this example it can be concluded that, at degenerate bifurcation of
second kind, several limit cycles may emerge. Of course, with the same idea,
it is possible to build polynomial dynamical systems with higher degree $%
2N+1 $ showing the emergence of $N$ different limit cycles at a super- or
subcritical degenerate HB of second kind. The frequency of the emerging
cycles is 
\begin{equation*}
\varpi =1+O\left( \left\vert \tau _{a}\right\vert ^{1/3}\right)
\end{equation*}%
as $N=1$ and $\gamma =2/3$ in Eq.\ref{frequencyDeg}.

\subsection{Semistable cycles in supercritical degenerate Hopf bifurcation
of second kind \label{subsection multiple II}}

The following system shows a supercritical degenerate bifurcation of second
kind:

\begin{equation}
\left( 
\begin{array}{c}
\overset{\cdot }{x} \\ 
\overset{\cdot }{y}%
\end{array}%
\right) =\left( 
\begin{array}{cc}
a^{3} & -1 \\ 
1 & a^{3}%
\end{array}%
\right) \left( 
\begin{array}{c}
x \\ 
y%
\end{array}%
\right) -2a^{2}\left( x^{2}+y^{2}\right) \left( 
\begin{array}{c}
x \\ 
y%
\end{array}%
\right) +a\left( x^{2}+y^{2}\right) ^{2}\left( 
\begin{array}{c}
x \\ 
y%
\end{array}%
\right) \text{ .}  \label{AnHB2}
\end{equation}%
The origin is an stable focus without any surrounding limit cycle while the
trace $\tau =2a^{3}$ is negative. At the critical value $a=0$, the origin is
clearly a center to the system Eq.\ref{AnHB2}. The system shows a degenerate
bifurcation of second kind in accordance with Definition \ref{definition
bifurcations}. Further, when the trace becomes positive the steady state
becomes unstable and, a single small limit cycle emerges with radius%
\begin{equation*}
r=a^{1/2}\text{ .}
\end{equation*}%
This limit cycle is the $\omega $-limit set of any orbit inside the circle
with the exception of the origin, but it is the $\alpha $-limit set of any
orbit outside the circle. The corresponding system in polar coordinates is:%
\begin{equation*}
\left\{ 
\begin{array}{c}
\overset{\cdot }{r}=ar\left( r^{2}-a\right) ^{2} \\ 
\overset{\cdot }{\theta }=1\text{ }%
\end{array}%
\right. \text{ .}
\end{equation*}%
From this example it can be concluded that, at a degenerate bifurcation,
semistable limit cycles may emerge. The radius of the limit cycle is $%
r=O_{S}\left( \tau ^{1/6}\right) $ as $\tau \rightarrow 0$. The frequency of
the emerging cycle is 
\begin{equation*}
\varpi =1+O\left( \left\vert \tau _{a}\right\vert ^{1/3}\right)
\end{equation*}%
as $N=1$ and $\gamma =2/3$ in Eq.\ref{frequencyDeg}.

\subsection{Degenerate Hopf bifurcation of first kind without emergence of
limit cycle \label{Subsect deg 1stkind}}

\begin{enumerate}
\item Let us first consider a system showing a degenerate HB of first kind: 
\begin{equation}
\left( 
\begin{array}{c}
\overset{\cdot }{x} \\ 
\overset{\cdot }{y}%
\end{array}%
\right) =\left( 
\begin{array}{cc}
a & -1 \\ 
1 & a%
\end{array}%
\right) \left( 
\begin{array}{c}
x \\ 
y%
\end{array}%
\right) -a\left( x^{2}+y^{2}\right) \left( 
\begin{array}{c}
x \\ 
y%
\end{array}%
\right)  \label{sysDegHB1}
\end{equation}%
The origin is an stable focus without limit cycle surrounding while $a<0$.
At the critical value $a=0$ the origin is a center, and the steady state
turns unstable if $a>0$. It can be noted the existence of a limit cycle with
radius $r=1$, but this cycle is not a consequence of the bifurcation because
it persist along the bifurcation. The polar system is: 
\begin{equation*}
\left\{ 
\begin{array}{c}
\overset{\cdot }{r}=ar\left( 1-r^{2}\right) \\ 
\overset{\cdot }{\theta }=1\text{ }%
\end{array}%
\right. \text{ .}
\end{equation*}%
showing that the cycle $r=1$ is a stable limit cycle for $a>0$, changing its
stability in dependence of the sign of $a$. This circle is still an orbit of
Eq.\ref{sysDegHB1} for $a=0$.

\item We recall in the fact that, the persisting limit cycle may be
semistable, as in the system%
\begin{equation}
\left( 
\begin{array}{c}
\overset{\cdot }{x} \\ 
\overset{\cdot }{y}%
\end{array}%
\right) =\left( 
\begin{array}{cc}
a & -1 \\ 
1 & a%
\end{array}%
\right) \left( 
\begin{array}{c}
x \\ 
y%
\end{array}%
\right) -2a\left( x^{2}+y^{2}\right) \left( 
\begin{array}{c}
x \\ 
y%
\end{array}%
\right) +a\left( x^{2}+y^{2}\right) ^{2}\left( 
\begin{array}{c}
x \\ 
y%
\end{array}%
\right)  \label{sysDegHB2}
\end{equation}%
which also shows a degenerate bifurcation of first kind. The corresponding
polar system is:%
\begin{equation*}
\left\{ 
\begin{array}{c}
\overset{\cdot }{r}=ar\left( 1-r^{2}\right) ^{2} \\ 
\overset{\cdot }{\theta }=1\text{ }%
\end{array}%
\right.
\end{equation*}%
so, the limit cycle $r=1$ is semistable, its interior stability changes with
the sign of $a$, and it is still an orbit of Eq.\ref{sysDegHB2} for $a=0$.
\end{enumerate}

\subsection{Degenerate HB of first kind showing limit cycles at infinity 
\label{cycles at infinity}}

Let us consider now a system showing a degenerate HB of first kind, but
leading in this case to the so called HB at infinity \cite{gasull}: 
\begin{equation}
\left( 
\begin{array}{c}
\overset{\cdot }{x} \\ 
\overset{\cdot }{y}%
\end{array}%
\right) =\left( 
\begin{array}{cc}
a & -1 \\ 
1 & a%
\end{array}%
\right) \left( 
\begin{array}{c}
x \\ 
y%
\end{array}%
\right) -a\ a^{\beta }\left( x^{2}+y^{2}\right) \left( 
\begin{array}{c}
x \\ 
y%
\end{array}%
\right)  \label{sysDegHB3}
\end{equation}%
$\beta >0$. Here we take $a^{\beta }=\mathtt{sign}\left( a\right) \cdot
\left\vert a\right\vert ^{\beta }$. The portrait becomes clear taking polar
coordinates:%
\begin{equation*}
\left\{ 
\begin{array}{c}
\overset{\cdot }{r}=ar\left( 1-a^{\beta }\ r^{2}\right) \\ 
\overset{\cdot }{\theta }=1\text{ }%
\end{array}%
\right. \text{ .}
\end{equation*}%
We conclude that, as a consequence of the degenerate bifurcation of the
system Eq.\ref{sysDegHB3} no limit cycle emerges surrounding the origin, but
an unstable limit cycle surrounding the point at infinity, whose radius is $%
r=a^{-\beta /2}$, emerges supercritically. More precisely, the cycle emerges
for small positive values of the trace $\tau =2a$.

\subsection{A sufficient condition for degenerate Hopf bifurcation of second
kind \label{subsection sufficient degenerate}}

We shall consider here a degenerate HB for the system Eq.\ref{variations} up
to $6$-jet-equivalence. The following statement, which is inspired in the
examples in Subsections \ref{subsection multiple} and \ref{subsection
multiple II}, gives a sufficient condition for the existence of degenerate
bifurcations of second kind. The idea in this assertion is very simple.

\begin{proposition}
Let Eq.\ref{complex eigenvalues cond} holds for the system Eq.\ref%
{variations} at the origin, and let $j_{6}\left( F\right) \left( 0\right) $
be such that the coefficients $p_{3}$ in Eq.\ref{HC3} and $p_{5}$ in Eq.\ref%
{HC5} satisfy the conditions:%
\begin{equation*}
p_{3}=\tau _{a}^{\gamma }\ Q_{3}+o\left( \tau _{a}^{\gamma }\right) \text{ ; 
}p_{5}=-\tau _{a}^{2\gamma -1}\ Q_{5}+o\left( \tau _{a}^{2\gamma -1}\right)
\end{equation*}%
for some $\gamma $ ($1/2<\gamma <1$) as $\tau _{a}\rightarrow 0$, where $%
Q_{3}$, $Q_{5}$ are both positive numbers and%
\begin{equation*}
\Delta =Q_{3}^{2}-4Q_{5}\geq 0\text{ .}
\end{equation*}%
Then, two different limit cycles if $\Delta >0$, or one semistable limit
cycle if $\Delta =0$, emerge at the supercritical degenerate bifurcation.
The radius of the cycles are $O_{S}\left( \tau _{a}^{\left( 1-\gamma \right)
/2}\right) $ as $\tau _{a}\rightarrow 0+$.
\end{proposition}

\begin{proof}
Taking $1/2<\gamma <1$ in Eq.\ref{degen2dkind} follows Eq.\ref{epsilon deg},
hence $\varepsilon =\tau _{a}^{\left( 1-\gamma \right) /2}$. Up to the $%
O\left( \tau _{a}\right) $ leading term in Eq.\ref{limit cycle existence},
we get%
\begin{equation*}
Q_{5}\ r^{4}-Q_{3}\ r^{2}+1=0
\end{equation*}%
meaning that the algebraic equation above has either two different positive
roots $\rho _{k}$ ($k=1,2$), or a single positive root with multiplicity
two, depending on $\Delta $. Then, the proof follows from Eq.\ref{limit
cycle existence} and Eq.\ref{r general}.
\end{proof}

\section{Conclusions}

To study degenerate Hopf bifurcation in smooth dynamical systems near an
isolated singular point on the plane we first focus in a classification of
such bifurcations via a discriminant function. If the bifurcation is
non-degenerate, then a single limit cycle emerges. The radius of this cycle
can not be a priori supposed to have order $O\left( \tau _{a}^{1/2}\right) $
as $\tau _{a}\rightarrow 0$, but in general have order $O\left( \tau
_{a}^{1/2N}\right) $, for some integer $N$. In this scenario, the emerging
limit cycle may coexist with other cycles which persist along the
bifurcation. This classification gives some light to the question of whether
degenerate bifurcations lead to the emergence of limit cycles. We show that,
no limit cycle surrounding the steady state emerges neither at super- nor at
subcritical HB when it is degenerate of first kind. In this scenario, limit
cycles surrounding the point at infinity may also emerge. For a degenerate
HB of second kind, we found the anomalous asymptotic order of the radius and
period of the emerging limit cycles, and further, we give a sufficient
condition to the appearance either of a couple of limit cycles or one
semistable cycle of the system Eq.\ref{variations}, up to $6$%
-jet-equivalence. Finally, we propose some examples, which are called here
\textquotedblleft typical\textquotedblright\ forms, showing different
behaviors that can occur at supercritical degenerate HB, being either of
first or of the second kind.

\section{Appendix\label{section appendix}}

In this Appendix we include the values of the integrals in Eq.\ref%
{integral(m,n)}, which are used in the implementation of the
Krylov-Bogoliubov averaging method. Further, we include the coefficients $%
H_{mn}$ leading to Hopf coefficients. Notice that, for $m+n\leq 6$ the
non-zero numbers are:%
\begin{equation*}
\begin{array}{cc}
K_{03}=\allowbreak \frac{3}{4}\pi \text{ ;} & K_{21}=\allowbreak \frac{1}{4}%
\pi \text{ ;}%
\end{array}%
\end{equation*}%
\begin{equation*}
\begin{array}{ccc}
K_{05}=\allowbreak \frac{5}{8}\pi \text{ ;} & K_{23}=\frac{1}{8}\pi \text{ ;}
& K_{41}=\frac{1}{8}\pi \text{ .}%
\end{array}%
\end{equation*}

For a given matrix $\Gamma $, which is a non-trivial linear combination of
the pair in Eq.\ref{gamma generators}, we denote $\Gamma ^{-1}=\left( \mu
_{ij}\right) $. Then, the $H_{mn}$ leading to Hopf coefficients up to $6$%
-jet-equivalence are

\QTP{Body Math}
\begin{eqnarray*}
H_{21} &=&3R_{30}\mu _{11}^{2}\mu _{12}+2R_{21}\mu _{11}\mu _{12}\mu
_{21}+R_{21}\mu _{11}^{2}\mu _{22} \\
&&+2R_{12}\mu _{11}\mu _{21}\mu _{22}+R_{12}\mu _{12}\mu
_{21}^{2}+3R_{03}\mu _{21}^{2}\mu _{22}\text{ ,}
\end{eqnarray*}%
\begin{equation*}
H_{03}=R_{12}\mu _{12}\mu _{22}^{2}+R_{03}\mu _{22}^{3}+R_{30}\mu
_{12}^{3}+R_{21}\mu _{12}^{2}\mu _{22}\text{ ,}
\end{equation*}

\QTP{Body Math}
\begin{eqnarray*}
H_{41} &=&2R_{32}\mu _{11}^{3}\mu _{21}\mu _{22}+5R_{50}\mu _{11}^{4}\mu
_{12}+\allowbreak R_{41}\mu _{11}^{4}\mu _{22}+4R_{41}\mu _{11}^{3}\mu
_{12}\mu _{21} \\
&&+2R_{23}\mu _{11}\mu _{12}\mu _{21}^{3}+3R_{23}\mu _{11}^{2}\mu
_{21}^{2}\mu _{22}+3R_{32}\mu _{11}^{2}\mu _{12}\mu _{21}^{2} \\
&&+R_{14}\mu _{12}\mu _{21}^{4}+4R_{14}\mu _{11}\mu _{21}^{3}\mu
_{22}+5R_{05}\mu _{21}^{4}\allowbreak \mu _{22}\text{ ,}
\end{eqnarray*}%
\begin{eqnarray*}
H_{23} &=&10R_{50}\mu _{11}^{2}\mu _{12}^{3}+4R_{41}\mu _{11}\mu
_{12}^{3}\mu _{21}+\allowbreak 6R_{41}\mu _{11}^{2}\mu _{12}^{2}\mu _{22} \\
&&+6R_{32}\mu _{11}\mu _{12}^{2}\mu _{21}\mu _{22}+R_{32}\mu _{12}^{3}\mu
_{21}^{2}+3R_{32}\mu _{11}^{2}\mu _{12}\mu _{22}^{2} \\
&&+6\allowbreak R_{23}\mu _{11}\mu _{12}\mu _{21}\mu _{22}^{2}+R_{23}\mu
_{11}^{2}\mu _{22}^{3}+\allowbreak 3R_{23}\mu _{12}^{2}\mu _{21}^{2}\mu _{22}
\\
&&+6R_{14}\mu _{12}\mu _{21}^{2}\mu _{22}^{2}+4R_{14}\mu _{11}\mu _{21}\mu
_{22}^{3}+10R_{05}\mu _{21}^{2}\mu _{22}^{3}\text{ ,}
\end{eqnarray*}%
\begin{eqnarray*}
H_{05} &=&R_{32}\mu _{12}^{3}\mu _{22}^{2}+R_{23}\mu _{12}^{2}\mu
_{22}^{3}+R_{14}\mu _{12}\mu _{22}^{4} \\
&&+R_{05}\mu _{22}^{5}+R_{50}\mu _{12}^{5}+\allowbreak R_{41}\mu
_{12}^{4}\mu _{22}\text{ ,}
\end{eqnarray*}%
where $R_{kl}$ are given in Eq.\ref{Rkl}.

\end{document}